\documentclass{amsart}
\usepackage{amsmath,amssymb,amsthm,multirow,hyperref,enumitem,tikz}
\usetikzlibrary{decorations.pathmorphing,decorations.pathreplacing,calc}
\newtheorem{lemma}{Lemma}[section]
\newtheorem{corollary}[lemma]{Corollary}

\newtheorem{theorem}[lemma]{Theorem}

\begin{document}

\title[Base size]{On the base size of a finite group on its action on the lattice of subgroups}  

\author[J-L.~Du]{Jia-Li Du}
\address{School of Mathematical Sciences, Nanjing Normal University, Nanjing, 210023, P.R. China, Ministry of Education Key Laboratory of NSLSCS, Nanjing, 210023, P.R. China}
\email{dujl@njnu.edu.cn}

\author[J.~Morris]{Joy Morris}
\address{Department of Mathematics and Computer Science\\
University of Lethbridge\\
Lethbridge, AB. T1K 3M4}    
\email{joy.morris@uleth.ca}
\thanks{The first author was supported by the National Natural Science Foundation of China (12571370, 12371025) and National Key R\&D Program of China (211070B62501). The second author was supported by the Natural Science and Engineering Research Council of Canada (grant RGPIN-2024-04013).}

\author[P.~Spiga]{Pablo Spiga}
\address{Dipartimento di Matematica e Applicazioni, University of Milano-Bicocca, Via Cozzi 55, 20125 Milano, Italy} 
\email{pablo.spiga@unimib.it}

\begin{abstract}
Given a finite group $R$, we investigate the base size of the action of the automorphism group of $R$ on the lattice of subgroups of $R$. Our main result shows that this base size is $1$ if and only if $R$ is cyclic. Our motivation arises from a conjecture of Babai on the problem of representing groups as automorphism groups of lattices with a bounded number of orbits.

\smallskip
\end{abstract}



\maketitle
\section{Introduction}\label{sec:intro}
Let $R$ be a finite group, let $\mathrm{Aut}(R)$ denote the automorphism group of $R$ and let $\mathrm{Sub}(R)$ be the collection of all subgroups of $R$. The group $\mathrm{Aut}(R)$ has a natural action on $\mathrm{Sub}(R)$, where given $H\in\mathrm{Sub}(R)$ and $\alpha\in \mathrm{Aut}(R)$, $H^\alpha$ is the image of $H$ under $\alpha$. In this paper, we are interested in the permutation group induced by the action of $\mathrm{Aut}(R)$ on $\mathrm{Sub}(R)$, which, for lack of a better name, we denote by $\mathrm{A}(R)$.

We recall that $\alpha\in\mathrm{Aut}(R)$ is said to be a \textit{\textbf{power automorphism}} if, for every $x\in R$, there exists $n_x\in\mathbb{Z}$ such that $x^\alpha=x^{n_x}$. The set of power automorphisms of $R$ forms a subgroup $\mathrm{PAut}(R)$ of $\mathrm{Aut}(R)$.

In general, the action of $\mathrm{Aut}(R)$ on $\mathrm{Sub}(R)$ is not faithful. Clearly, every power automorphism fixes setwise every subgroup of $R$, and hence $\mathrm{PAut}(R)$ is contained in the kernel of the action of $\mathrm{Aut}(R)$ on $\mathrm{Sub}(R)$. Conversely, if $\alpha\in \mathrm{Aut}(R)$ acts trivially on the set $\mathrm{Sub}(R)$, then $\alpha$ fixes setwise each cyclic subgroup $\langle x\rangle$ of $R$, and hence $x^\alpha\in\langle x\rangle$ for every $x\in R$, that is, $\alpha$ is a power automorphism. Summing up, $\mathrm{PAut}(R)$ is the kernel of the action of $\mathrm{Aut}(R)$ on $\mathrm{Sub}(R)$ and hence 
$$\mathrm{A}(R)\cong \frac{\mathrm{Aut}(R)}{\mathrm{PAut}(R)}.$$

We recall that a \textit{\textbf{base}} for a permutation group $G$ is a sequence $\omega_1 , \ldots, \omega_\ell$ of elements in the domain of $G$ whose pointwise stabilizer in $G$ is the identity. The minimal cardinality of a base is called the \textit{\textbf{base size}}.

We are interested in the base size of the action of $\mathrm{Aut}(R)$ on $\mathrm{Sub}(R)$; that is, in the base size of $\mathrm{A}(R)$.  We determine the finite groups $R$ such that this base size is $1$. When $R$ is cyclic, $\mathrm{A}(R)=1$ because $R$ itself is a base for this action, since every automorphism of $R$ is a power automorphism. Hence, $\mathrm{A}(R)$ has base size $1$. Our main result shows the converse.

\begin{theorem}\label{thrm:main}
Let $R$ be a finite group. The base size of the action of $\mathrm{Aut}(R)$ on $\mathrm{Sub}(R)$ is $1$ if and only if $R$ is cyclic.
\end{theorem}

We find this result quite unexpected, because it shows that $\mathrm{Aut}(R)$ can have base size $1$ in the action under consideration only when the action is trivial. Therefore we have the following corollary.

\begin{corollary}\label{cor:main}
Let $R$ be a finite group. In the action of $\mathrm{Aut}(R)$ on $\mathrm{Sub}(R)$ there exists a regular orbit (that is, an orbit on which the permutation group induced by $\mathrm{Aut}(R)$ acts regularly) if and only if $R$ is cyclic.
\end{corollary}

\section{Our motivation}\label{sec:motivation}
Our motivation for this investigation arises from an apparently different question. Since the early 1900s, there has been considerable interest in determining, for a given class of combinatorial objects, whether every finite group $R$ can be realized as the automorphism group of some object in that class. Perhaps the most famous example is the theorem of Frucht~\cite{frucht}, which shows that, for every finite group $R$, there exists a graph $\Gamma$ such that $R\cong\mathrm{Aut}(\Gamma)$. The analogous result for lattices is essentially due to Birkhoff~\cite{birkhoff}: for every finite group $R$, there exists a lattice $L$ such that $R\cong\mathrm{Aut}(L)$. Similar results hold for other classes of combinatorial objects.

Once such representation results are established, it is natural to refine the constructions in order to control additional combinatorial properties of the objects involved. The aim is not only to realize a given abstract group as an automorphism group, but also to ensure that the object has a “small” or simple structure, so that the group reflects more closely the combinatorial nature of the object.

We illustrate this with two classical examples. In the case of graphs, one may ask, given a finite group $R$, to find a graph $\Gamma$ such that $R\cong\mathrm{Aut}(\Gamma)$ and such that the number of vertex-orbits of $\mathrm{Aut}(\Gamma)$ is as small as possible. (The construction of Frucht depends on the number of generators of $R$.) In this context, the notion of a \textit{\textbf{graphical regular representation}} arises: this is a graph $\Gamma$ such that $R\cong\mathrm{Aut}(\Gamma)$ and $\mathrm{Aut}(\Gamma)$ acts regularly on the vertex set.

A related line of investigation considers representations of finite groups via automorphism groups of graphs $\Gamma$ such that $\mathrm{Aut}(\Gamma)$ has a bounded number of edge-orbits. This problem was motivated by a conjecture of Babai. (Again, Frucht’s construction, via suitable modifications of Cayley graphs, produces graphs whose automorphism groups have a number of edge-orbits bounded above by the number of generators of the group.) Somewhat surprisingly, the answer to this question is negative, as shown by Goodman~\cite{Goodman}: for every positive integer $\ell$, there exists a finite group $R_\ell$ such that, if $R_\ell\cong\mathrm{Aut}(\Gamma)$ for some finite graph $\Gamma$, then $\mathrm{Aut}(\Gamma)$ has at least $\ell$ edge-orbits. The methods developed in obtaining this result (see, for instance,~\cite{Babai1}) are of independent interest.

In this line of research, a major open problem concerns the analogous question for lattices. This problem was posed by Babai in several places; see, for example,~\cite[see abstract]{5},~\cite[Problem~5.27]{4},~\cite[Conjecture~4.13]{Babai0} and~\cite[see page~8]{Babai1}. The question is the following:
\begin{center}
Is there an absolute constant $c$ such that, for every finite group $R$, there exists a lattice $L$ with $R\cong\mathrm{Aut}(L)$ and such that $\mathrm{Aut}(L)$ has at most $c$ orbits?
\end{center}
Babai conjectured that the answer is affirmative. To date, there has been almost no progress on this conjecture, and there are good reasons for this. Indeed, given a finite group $R$, Birkhoff’s construction of a lattice $L$ proceeds via a graph, that is, the lattice $L$ is built from a graph. Through this standard construction, finite groups that can be realized as automorphism groups of graphs with a bounded number of edge-orbits yield lattices whose automorphism groups have a bounded number of orbits. However, as noted above, there exist groups that require arbitrarily many edge-orbits in any such graph representation. Therefore, this natural strategy for approaching Babai’s conjecture appears to lead to a dead end.

The work presented in this paper is inspired by the comments and questions in~\cite[Section~8]{Babai1}. These problems are twofold in nature. On the one hand, they ask for the minimal number of subgroups required so that the identity is the only automorphism fixing these subgroups setwise. On the other hand, they consider the analogous question when one is allowed to choose a bounded number of subgroups together with a bounded number of group elements, and again ask whether the identity is the only automorphism fixing them.

The scope of these questions is at least threefold: first, to improve the results obtained in~\cite{Babai1}; second, to characterize, in purely algebraic terms, the groups admitting a graphical representation with a bounded number of edge-orbits; and third, to develop ideas that may have further applications, particularly in the context of lattices.

Roughly speaking, some of the questions in~\cite[Section~8]{Babai1} can be interpreted in terms of the base size of $\mathrm{Aut}(R)$ in its natural action on $\mathrm{Sub}(R)$, that is, the base size of $\mathrm{A}(R)$. It is through the study of~\cite{Babai1}, and in connection with attempts to approach Babai's conjecture on lattices, that we arrived at the result stated in Theorem~\ref{thrm:main}.

Although Theorem~\ref{thrm:main} does not appear to have direct applications to Babai's conjecture on lattices, in the course of our exploration of this problem we have also developed methods (not included in this paper) that allow us to extend some of the results in~\cite{Babai1}, especially in the case of groups of odd order.

\section{Preliminaries}\label{sec:preliminaries}
The power automorphisms and their collection $\mathrm{PAut}(R)$ are related to the concept of the \textit{\textbf{Baer norm}}, see~\cite{Baer1}. Indeed, we let 
$$N(R)=\bigcap_{H\le R}{\bf N}_R(H)$$
be the set of elements of $R$ that normalize every subgroup of $R$. If we let $\mathrm{Inn}:R\to\mathrm{Aut}(R)$ be the mapping defined by $x\mapsto \mathrm{Inn}(x)$, where $\mathrm{Inn}(x)$ is the automorphism of $R$ given by conjugation by $x$, and we abuse notation slightly by considering $\mathrm{Inn}(N(R))$ to be a collection of automorphisms of $R$ (induced by conjugation by elements of $N(R)$) rather than automorphisms of $N(R)$, then 
\begin{equation}\label{eq:1}
\mathrm{Inn}(R)\cap \mathrm{PAut}(R)=\mathrm{Inn}(N(R)).
\end{equation}
Roughly speaking, the inner automorphisms that are power automorphisms arise by conjugation by elements of the Baer norm $N(R)$.

The structure of the Baer norm $N(R)$ of $R$ is rather restricted. Recall that a group is said to be a \textit{\textbf{Dedekind group}} if all of its subgroups are normal. Since the elements of $N(R)$ normalize each subgroup of $R$, every subgroup of $N(R)$ is normal in $N(R)$, and hence $N(R)$ is a Dedekind group. Dedekind~\cite{Dedekind} has shown that a Dedekind group is either abelian, or of the form $Q_8\times E\times O$, where $Q_8$ is the quaternion group of order $8$, $E$ is an elementary abelian $2$-group, and $O$ is an abelian group of odd order.

One of the main results concerning the Baer norm is a result of Schenkman~\cite{Schenkman}. To state this result we need some basic definitions. We let ${\bf Z}(R)$ denote the \textit{\textbf{center}} of $R$. This definition allows us to define an ascending series of subgroups of $R$, namely, given $n\in\mathbb{N}$, let ${\bf Z}_0(R)=1$, ${\bf Z}_1(R)={\bf Z}(R)$ and, for $n\ge 2$, let ${\bf Z}_n(R)$ be the subgroup of $R$ such that ${\bf Z}(R/{\bf Z}_{n-1}(R))={\bf Z}_n(R)/{\bf Z}_{n-1}(R)$. Moreover, the \textit{\textbf{derived subgroup}} $R'$ of $R$ is the subgroup of $R$ generated by the \textit{\textbf{commutators}} of $R$, that is, by the expressions $[x,y]=x^{-1}y^{-1}xy$, for $x,y\in R$. Finally, for a subgroup $H$ of $R$, $\mathbf{C}_R(H)$ is the centralizer of $H$ in $R$.

\begin{theorem}[{Schenkman~\cite[Theorem]{Schenkman}}]\label{thrm:schenkman}
We have $R'\le {\bf C}_R(N(R))$ and $N(R)\le {\bf Z}_2(R)$.
\end{theorem}

 When $\mathrm{Aut}(R)=\mathrm{PAut}(R)$ then $\mathrm{A}(R)=1$, so clearly the base size of $A(R)$ is $1$. Therefore it is relevant to observe the following.

\begin{lemma}\label{l:1}
We have $\mathrm{Aut}(R)=\mathrm{PAut}(R)$ if and only if $R$ is cyclic.
\end{lemma}

We postpone the proof of Lemma~\ref{l:1} because we first prove an auxiliary result (needed in the proof of Lemma~\ref{l:1}) that will be useful also later.

\begin{lemma}\label{l:-1}
Let $R$ be a Dedekind group. If $R$ is not cyclic, then for every subgroup $H$ of $R$ there exists a non-power automorphism of $R$ fixing $H$ setwise.
\end{lemma}

\begin{proof}
We use the classification of Dedekind~\cite{Dedekind} for Dedekind groups. Suppose first that $R=Q_8\times E\times O$, where $Q_8$ is the quaternion group of order $8$, $E$ is an elementary abelian $2$-group and $O$ is an abelian group of odd order. Let $H$ be a subgroup of $R$. Then, replacing $H$ by a suitable $\mathrm{Aut}(R)$-conjugate, we have
$H=H_1\times H_2\times H_3$, where $H_1\le Q_8$, $H_2\le E$ and $H_3\le O$.
We let $i,j,k,-i,-j,-k,-1,1$ denote the elements of $Q_8$. Representatives for the $\mathrm{Aut}(Q_8)$-classes of subgroups of $Q_8$ are $1,\langle -1\rangle$, $\langle i\rangle$ and $Q_8$. Therefore, replacing $H_1$ by a suitable $\mathrm{Aut}(Q_8)$-conjugate, we may suppose that $H_1$ is one of these four subgroups. Now, $Q_8$ admits an automorphism $\varphi$ with $i^\varphi=i$, $j^\varphi=k$ and $k^\varphi=-j$. Clearly, $\varphi$ is not a power automorphism (because $k=j^\varphi\notin\langle j\rangle$) and it fixes setwise $1,\langle -1\rangle$, $\langle i\rangle$ and $Q_8$. This automorphism $\varphi$ extends to an automorphism of $R=Q_8\times E\times O$ fixing $H$ setwise.

Assume then that $R$ is abelian and non-cyclic. Let $H$ be a subgroup of $R$. For each prime $p$, let $R_p$ be the Sylow $p$-subgroup of $R$. Thus $R=\prod_p R_p$ and $H=\prod_p (H\cap R_p)$. In particular, there exists a prime $p$ such that $R_p$ is not cyclic. Using the structure theorem of finitely generated abelian groups, we may write 
\begin{align*}
R_p&=\langle a_1\rangle\times \cdots\times\langle a_\ell\rangle,\\
H\cap R_p&=\langle a_1^{p^{f_1}}\rangle\times \cdots\times\langle a_\ell^{p^{f_\ell}}\rangle,
\end{align*}
where $\ell\ge 2$, and taking $e_i$ to be the nonnegative integer such that $\mathbf{o}(a_i)=p^{e_i}$, we have $e_i\ge f_i \ge 0$ and $e_i\ge 1$ for every $i\in \{1,\ldots,\ell\}$. 

If $H \cap R_p \neq R_p$ then there is some $i$ such that $f_i>0$ and some $1\le j \le \ell$ with $j \neq i$; fix such an $i$ and $j$. Otherwise, take $i=1$ and $j=2$. Let $\varphi$ be the homomorphism of $R_p$ into itself mapping $a_k$ to $a_k$ when $k\ne i$ and mapping $a_i$ to $a_i a_j^{p^{e_j-1}}$. 
Clearly, $\varphi$ is an automorphism of $R_p$ because it is a surjective homomorphism. Moreover, $\varphi$ fixes $H\cap R_p$ setwise: if $H\cap R_p=R_p$ then this is obvious; otherwise since $f_i \ge 1$,
 $$(a_i^\varphi)^{p^{f_i}}=a_i^{p^{f_i}}a_j^{p^{e_j-1+f_i}}=a_i^{p^{f_i}}\in H\cap R_p.$$
Furthermore, $\varphi$ is a non-power automorphism of $R_p$ because $a_i^{-1}a_i^\varphi=a_j^{p^{e_j-1}}\notin\langle a_i\rangle$.

Now, $\varphi$ extends to a group automorphism $\phi$ of $R$ by fixing $R_q$ element-wise, for every $q\ne p$. Hence, $\phi$ is a non-power automorphism fixing $H$ setwise. 
\end{proof}

\begin{proof}[Proof of Lemma~$\ref{l:1}$]
Clearly, if $R$ is cyclic, then $\mathrm{Aut}(R)=\mathrm{PAut}(R)$. Conversely, suppose $\mathrm{Aut}(R)=\mathrm{PAut}(R)$. In particular, $\mathrm{Inn}(R) \le \mathrm{PAut}(R)$, so by~\eqref{eq:1}, we have $\mathrm{Inn}(N(R))=\mathrm{Inn}(R)\cap \mathrm{PAut}(R)=\mathrm{Inn}(R)$. Let $x \in R$ be arbitrary. Then this says $\mathrm{Inn}(x)=\mathrm{Inn(y)}$ for some $y \in N(R)$. By definition of $N(R)$, $y$ normalizes every subgroup $H$ of $R$ and therefore so must $x$; hence $R=N(R)$. In particular, $R$ is a Dedekind group. Now the proof follows from Lemma~\ref{l:-1}. 
\end{proof}

\section{Proof of Theorem~$\ref{thrm:main}$}\label{sec:proof}
The rest of the paper is devoted to the proof of Theorem~\ref{thrm:main}. Since Lemma~\ref{l:1} shows one implication, we may assume that the base size of the action of $\mathrm{Aut}(R)$ on $\mathrm{Sub}(R)$ is $1$ and we need to prove that $R$ is cyclic. Since the base size is $1$, there exists a subgroup $H$ of $R$ such that
\begin{equation}\label{eq:2}
{\bf N}_{\mathrm{Aut}(R)}(H)=\mathrm{PAut}(R).
\end{equation}
Now, by definition of $N(R)$, $N(R) \le \mathbf{N}_R(H)$, so~\eqref{eq:1} and~\eqref{eq:2} imply
\begin{equation}\label{eq:3}
N(R)={\bf N}_R(H).
\end{equation}

We single out the first step of the proof in a separate result, because we believe it is of independent interest.

\begin{lemma}\label{l:2}
Let $R$ be a finite group and let $H$ be a subgroup of $R$ such that $N(R)={\bf N}_R(H)$. Then $R$ is nilpotent of nilpotency class at most $3$.
\end{lemma}

\begin{proof}
From Theorem~\ref{thrm:schenkman}, the hypothesis $N(R)=\mathbf{N}_R(H)$, and the fact that $H\le {\bf N}_R(H)$, we deduce the series of inequalities
$$R'\le {\bf C}_R(N(R))={\bf C}_R({\bf N}_R(H))\le {\bf C}_R(H)\le {\bf N}_R(H)=N(R)\le {\bf Z}_2(R).$$
This shows that $[R,R,R]=[R',R]\le [{\bf Z}_2(R),R]\le {\bf Z}(R)$ and hence $[R,R,R,R]\le [{\bf Z}(R),R]=1$.
\end{proof}
We now continue with the proof of Theorem~\ref{thrm:main}. From~\eqref{eq:3} and Lemma~\ref{l:2}, we deduce that $R$ is nilpotent of class at most $3$. Therefore, 
$$R=\prod_p R_p,$$
where $R_p$ is the Sylow $p$-subgroup of $R$ and $R_p$ has nilpotency class at most $3$. Since each $R_p$ is characteristic, we have
\begin{align*}
\mathrm{Aut}(R)&=\mathrm{Aut}\left(\prod_p R_p\right)=\prod_p \mathrm{Aut}(R_p),\\
\mathrm{PAut}(R)&=\mathrm{PAut}\left(\prod_p R_p\right)=\prod_p \mathrm{PAut}(R_p).
\end{align*}
In particular,~\eqref{eq:1} implies $\mathrm{PAut}(R_p)={\bf N}_{\mathrm{Aut}(R_p)}(H\cap R_p)$ for every prime $p$, and hence for the rest of the proof we may assume, without loss of generality, that $R$ itself is a $p$-group for some prime $p$. Moreover, we may suppose that $R$ has nilpotency class at most $3$. 

We argue now by contradiction and suppose that $R$ is not cyclic. 

We claim that
\begin{equation}\label{eq:4}
N(R)<R.
\end{equation}
If $R=N(R)$, then $R$ is a Dedekind group, and hence we obtain a contradiction from Lemma~\ref{l:-1}. This establishes~\eqref{eq:4}.

The following two technical lemmas will be used repeatedly during the remainder of our proof.

\begin{lemma}\label{chi}
Suppose $A\le K \le R$ and $A \le \mathbf{Z}(R)$. Let $\chi: R \to A$ be a homomorphism with kernel $K$, and define $\phi_\chi: R \to  R$ by $x^{\phi_\chi}=x\chi(x)$ for every $x \in R$. Then $\phi_\chi$ is a group automorphism. Furthermore, if $H \le K$ then $\phi_\chi$ is a power automorphism.
\end{lemma}

\begin{proof}
Since $\chi(x) \in A \le \mathbf{Z}(R)$, we have
$$x^{\phi_\chi}y^{\phi_\chi}= x\chi(x)y\chi(y)=xy\chi(x)\chi(y)=xy\chi(xy)=(xy)^{\phi_\chi}$$
for every $x,y \in R$, so $\phi_\chi$ is a homomorphism. Let $x \in \mathrm{Ker}(\phi_\chi)$. Then $1=x^{\phi_\chi}=x\chi(x)$ and hence $x^{-1}=\chi(x) \in A$. Since $A \le K=\mathrm{Ker}(\chi)$, we have $\chi(x)=1$, forcing $x=1$. Therefore $\mathrm{Ker}(\phi_\chi)=1$, so $\phi_\chi$ is a group automorphism as claimed.

If $H \le K$ then $\phi_\chi$ fixes each element of $H$ and hence fixes $H$ setwise. By our choice of $H$, this means $\phi_\chi$ is a power automorphism.
\end{proof}

\begin{lemma}\label{prime-aut}
Suppose $y \in \mathbf{Z}_2(R)\setminus \mathbf{Z}(R)$ and $|[R,y]|=p$ is prime. If $p$ is odd, further suppose $\mathbf{o}(y)=p$, while if $p=2$, suppose that $y^2=z$ is the involution of $\mathbf{Z}(R)$.

Fix $t \in R \setminus \mathbf{C}_R(y)$, and for every $c \in \mathbf{C}_R(y)$ and $x \in R$, if $x=t^ic$ with $0 \le i \le p-1$ then define $\phi(t^ic)=(yt)^ic$. Then $\phi$ is an automorphism of $R$.
\end{lemma}

\begin{proof}
Define $\varphi: R \to R$ by $x^\varphi=[x,y]$ for every $x \in R$, and observe that $\varphi$ is a homomorphism whose image is $[R,y]$ and whose kernel is $\mathbf{C}_R(y)$. Thus $R/\mathbf{C}_R(y) \cong [R,y]$ is cyclic of order $p$, so every element of $R$ can be written uniquely as $t^ic$ for some $0 \le i \le p-1$ and some $c \in \mathbf{C}_R(y)$.

Arguing inductively on $i$, using the fact that $y$ commutes with every element of $\mathbf{C}_R(y)$ and that $c \in \mathbf{C}_R(y)$ implies $c^{t^i} \in \mathbf{C}_R(y)$ since $R/\mathbf{C}_R(y)$ is cyclic, we get $$c^{t^i}=c^{(yt)^i}.$$

We can now show that $\phi$ is a homomorphism. Let $c_1, c_2 \in \mathbf{C}_R(y)$ and $0 \le i_1, i_2 <p$ and suppose $i_1+i_2<p$. Then 
\begin{align*}
(t^{i_1}c_1 \cdot t^{i_2}c_2)^\phi&=(t^{i_1+i_2}c_1^{t^{i_2}}c_2)^\phi=(yt)^{i_1+i_2}c_1^{t^{i_2}}c_2\\
&=(yt)^{i_1+i_2}c_1^{(yt)^{i_2}}c_2=(yt)^{i_1}c_1(yt)^{i_2}c_2=(t^{i_1}c_1)^\phi(t^{i_2}c_2)^\phi.
\end{align*}

In order to deal with the case $i_1+i_2 \ge p$ we need an additional observation. Standard commutator computations give
$$(yt)^p=y^pt^p[t,y]^{\binom{p}{2}}.$$
If $p$ is odd then by hypothesis $y^p=1$; also $p$ divides $\binom{p}{2}$ and since $[t,y] \in [R,y]$ which has order $p$, this is $t^p$. If $p=2$ then this is $y^2[t,y]t^p$ and by hypothesis $[t,y]=y^2=z$ which has order $2$, so again we have $t^p$.

Now assume $i_1+i_2 \ge p$, say $i_1+i_2=p+i_3$ where $0 \le i_3 <p$. Then using our new observation together with calculations similar to the previous case, we have
\begin{align*}
(t^{i_1}c_1 \cdot t^{i_2}c_2)^\phi&=(t^{i_1+i_2}c_1^{t^{i_2}}c_2)^\phi=(t^{i_3}t^pc_1^{t^{i_2}}c_2)^\phi\\
&=(yt)^{i_3}t^pc_1^{t^{i_2}}c_2=(yt)^{i_3}(yt)^pc_1^{t^{i_2}}c_2\\
&=(yt)^{i_1+i_2}c_1^{(yt)^{i_2}}c_2=(yt)^{i_1}c_1(yt)^{i_2}c_2=(t^{i_1}c_1)^\phi(t^{i_2}c_2)^\phi.
\end{align*}

Since $t^ic$ and $(yt)^ic$ lie in the same coset of $\mathbf{C}_R(y)$, it is clear that the kernel of $\phi$ is trivial, so $\phi$ is an automorphism.
\end{proof}

Let $A$ be a subgroup of ${\bf Z}(R)$ and let $K$ be a subgroup of $R$ with $|R:K|=|A|=p$ and $AH\le K$ (since $R$ is not cyclic we cannot have $H=R$, so finding such a pair $(A,K)$ is possible). Let $\chi:R\to A$ be any surjective group homomorphism with kernel $K$: the existence of $\chi$ is clear because $R/K\cong A$; thus we have $p-1$ choices for an isomorphism from $R/K$ to $A$, and each such isomorphism can be lifted to a homomorphism $\chi$ from $R$ to $A$ with kernel $K$. By Lemma~\ref{chi}, the map $\phi_\chi$ defined by $x^{\phi_\chi}=x\chi(x)$ for every $x\in R$ is a power automorphism.

Assume now that ${\bf Z}(R)$ admits two distinct subgroups $A_1,A_2$ and that $R$ admits a subgroup $K$ with $|R:K|=|A_1|=|A_2|=p$ and $A_1A_2H\le K$. By Lemma~\ref{chi}, $\phi_{\chi_1}$ and $\phi_{\chi_2}$ are both power automorphisms of $R$. Let $x\in R\setminus K$. Then $x\chi_i(x)=x^{\phi_{\chi_i}}\in \langle x\rangle$, and hence $\chi_i(x)\in \langle x\rangle$. As $x$ is an arbitrary element outside $\mathrm{Ker}(\chi_i)$ and $\chi_i:R \to A_i$, we deduce $A_i\le\langle x\rangle$ and hence $A_1A_2\le\langle x\rangle$, which is a contradiction because $\langle x\rangle$ is cyclic whereas $A_1A_2$ is not. This contradiction arises from assuming the existence of $A_1,A_2$ and $K$ as above. Therefore, there are two alternatives: either ${\bf Z}(R)$ is cyclic (which excludes the existence of two distinct subgroups of order $p$ in ${\bf Z}(R)$), or ${\bf Z}(R)H=R$. In the latter case, $R={\bf N}_R(H)=N(R)$ and we contradict~\eqref{eq:4}. We have shown that
\begin{equation}\label{eq:5}
{\bf Z}(R) \hbox{ is cyclic}.
\end{equation}

We claim that
\begin{equation}\label{eq:6}
N(R) \hbox{ is abelian}.
\end{equation}
Suppose that $N(R)$ is non-abelian. Then, by the theorem of Dedekind, $N(R)=Q\times E$, where $Q\cong Q_8$ and $E$ is an elementary abelian $2$-group. Let $\langle z\rangle={\bf Z}(Q)$. Since $\langle z \rangle$ is characteristic in $N(R) \triangleleft R$, we have $\langle z \rangle \triangleleft R$. Since $R$ is a $p$-group, every normal subgroup has nontrivial intersection with $\mathbf{Z}(R)$, so $\langle z \rangle \le \mathbf{Z}(R)$. Since $\mathbf{Z}(R)$ is cyclic and ${\bf Z}(R)\le N(R)$, we have ${\bf Z}(R)=\langle z\rangle$. 

Observe that $H \le \mathbf{N}_R(H)=N(R)$. Suppose that ${\bf Z}(R)\le H$. For every $x \in R$ and $h \in H$ we have $[h,x]\in \mathbf{Z}(R)$ because $h \in \mathbf{Z}_2(R)$ from Theorem~\ref{thrm:schenkman}.  Thus $h^x=h[h,x] \in H$ since $h, [h,x] \in H$, so $H\unlhd R$. This implies $N(R)={\bf N}_R(H)=R$, and as usual we contradict~\eqref{eq:4}. So we must have ${\bf Z}(R)\nleq H$ and hence ${\bf Z}(R)\cap H=1$. From the structure of $N(R)=Q\times E$, we deduce that $H$ is an elementary abelian $2$-group and $H\le {\bf Z}(N(R))$.

Let $y\in N(R)$ with ${\bf o}(y)=4$. From the structure of $N(R)$, we deduce that $y^2=z$. We also have $y \in \mathbf{Z}_2(R)\setminus\mathbf{Z}(R)$, so since $[R,y] \le \mathbf{Z}(R)$ which has order 2, the conditions of Lemma~\ref{prime-aut} are met with $p=2$, and with $t \in R\setminus \mathbf{C}_R(y)$ fixed, the map $\phi:R\to R$ defined by $\phi(tc)=ytc$ for every $c \in \mathbf{C}_R(y)$ is a group automorphism. 
Since $H\le {\bf Z}(N(R))$, we have $H\le {\bf C}_R(y)$ and hence $\phi$ fixes $H$ pointwise. Thus $\phi$ is a power automorphism. However, $yt=t^\phi\notin\langle t\rangle$, because $t$ does not commute with $y$. This contradiction establishes~\eqref{eq:6}. 

Since $N(R)$ is abelian by~\eqref{eq:6}, by the structure theorem for finitely generated abelian groups, we may write
\begin{equation}\label{eq:aux}N(R)=\langle a_1\rangle\times\cdots\times\langle a_\ell\rangle,
\end{equation}
where ${\bf o}(a_i)=p^{n_i}$, $n_i\ge 1$ and, after relabeling the indices if necessary, we may assume ${\bf Z}(R)\le \langle a_1\rangle$. 
If $\ell=1$, then $N(R)$ is cyclic and hence every subgroup of $N(R)$ is characteristic. As $N(R)$ is characteristic in $R$, this implies that $H$ is a characteristic subgroup of $R$. Thus, by~\eqref{eq:2}, we have
$$\mathrm{PAut}(R)={\bf N}_{\mathrm{Aut}(R)}(H)=\mathrm{Aut}(R)$$
and hence $R$ is cyclic by Lemma~\ref{l:1}, contradicting our hypothesis on $R$. Therefore,
\begin{equation}\label{eq:6+}
\ell\ge 2.
\end{equation}

We claim that
\begin{equation}\label{eq:7}
p=2.
\end{equation}
Assume towards a contradiction that $p$ is odd. Set $y=a_\ell^{p^{n_\ell-1}}$. We have $y\in N(R)\le {\bf Z}_2(R)$ and $y\notin \langle a_1\rangle$, so $y\notin{\bf Z}(R)$. 
For every $x\in R$, since $[x,y]\in {\bf Z}(R)$ and $y^p=1$, we have $$[x,y]^{p}=[x,y^p]=[x,1]=1.$$  This yields  $|[R,y]|\le p$ and so $|[R,y]|=p$. The conditions of Lemma~\ref{prime-aut} are now satisfied, so with $t \in R\setminus\mathbf{C}_R(y)$ fixed, the map $\phi:R\to R$ defined by $\phi(t^ic)=(yt)^ic$ for every $c \in \mathbf{C}_R(y)$ and every $0 \le i <p$ is a group automorphism. 

Since $N(R)$ is abelian and $h \in N(R)$, we have $H\le N(R)\le {\bf C}_R(h)$ and hence $\phi$ fixes $H$ pointwise. Thus $\phi$ is a power automorphism. However, $h t=t^\phi\notin\langle t\rangle$, because $t$ does not commute with $h$. This contradiction establishes~\eqref{eq:7}.

We claim that
\begin{equation}\label{eq:8}
|{\bf Z}(R)|=2.
\end{equation}
We argue by contradiction and suppose that there exists $\bar c\in{\bf Z}(R)$ with ${\bf o}(\bar c)= 4$. Now, we mimic previous arguments, with a slight twist. Set $h=a_\ell^{2^{n_\ell-1}}$ and $y=\bar{c}h$. As before (with $\bar{c} \in \mathbf{Z}(R)$ not impacting these matters), $h$ (and therefore $y$) is not in $\mathbf{Z}(R)$, but is in $\mathbf{Z}_2(R)$. 
Let $t\in R\setminus {\bf C}_R(y)$. As $[t,y]\ne 1$, $|[R,y]|=2$ and $[R,y]\le{\bf Z}(R)$, we have $[t,y]=z$, where $z$ is the involution in the cyclic group ${\bf Z}(R)$. Thus $\bar c^2=z$, so $y^2=h^2\bar{c}^2=z$. 
So the conditions of Lemma~\ref{prime-aut} are met with $p=2$, and with $t \in R\setminus \mathbf{C}_R(y)$ fixed, the map $\phi:R\to R$ defined by $\phi(tc)=ytc$ for every $c \in \mathbf{C}_R(y)$ is a group automorphism. 
Since $N(R)$ is abelian, we have $H\le N(R)\le {\bf C}_R(y)$ and hence $\phi$ fixes $H$ pointwise. Thus $\phi$ is a power automorphism. However, $y t=t^\phi\notin\langle t\rangle$, because $t$ does not commute with $y$. This contradiction establishes~\eqref{eq:8}.

We claim that 
\begin{equation}\label{eq:9}
N(R) \hbox{ is an elementary abelian 2-group}.
\end{equation}
We use the direct product decomposition in~\eqref{eq:aux}. For every $i\in \{2,\ldots,\ell\}$, we have $[R,a_i]\le {\bf Z}(R)$, because $a_i\in N(R)\le {\bf Z}_2(R)$ by Theorem~\ref{thrm:schenkman}. By~\eqref{eq:8}, for every $x\in R$, we have $1=[x,a_i]^2=[x,a_i^2]$ and hence $a_i^2\in \langle a_i\rangle\cap {\bf Z}(R)\le \langle a_i\rangle\cap\langle a_1\rangle=1$. This implies $a_i^2=1$ and hence ${\bf o}(a_i)=2$. 
The same argument applied with $i=1$ gives ${\bf o}(a_1)\le 4$. Suppose that ${\bf o}(a_1)=4$. Then $|R:{\bf C}_R(a_1)|=|[R,a_1]|=2$ and $a_1^2=z$ is the involution of ${\bf Z}(R)$. Arguing as before, Lemma~\ref{prime-aut} with $y=a_1$ produces a non-power automorphism of $R$ fixing $H$ setwise, which is a contradiction. Therefore ${\bf o}(a_1)=2$ and this establishes~\eqref{eq:9}.

We claim that 
\begin{equation}\label{eq:10}
x^2\in N(R)\hbox{ and }{\bf o}(x)=4, \quad \forall x\in R\setminus N(R).
\end{equation}
From~\eqref{eq:9}, we have $y^2=1$ for every $y\in N(R)$. Let $h_1,\ldots,h_\kappa$ be generators of $H$. As $N(R)$ is abelian and $H \le N(R)={\bf N}_R(H)$, we have
$${\bf N}_R(H)={\bf C}_R(H)=\bigcap_{i=1}^\kappa {\bf C}_R(h_i).$$
As $[R,h_i]\le {\bf Z}(R)$ (again this follows from Theorem~\ref{thrm:schenkman} and $H \le N(R)$), we have $|R:{\bf C}_R(h_i)|=|[R,h_i]|\le 2$. Since $R/N(R)=R/\bigcap_i{\bf C}_R(h_i)$ embeds into $\prod_i (R/{\bf C}_R(h_i))$, we deduce that $R/N(R)$ is an elementary abelian $2$-group. Since $N(R)$ is also an elementary abelian $2$-group, each element of $R$ has order at most $4$.

Assume that there exists $t\in R\setminus N(R)$ with ${\bf o}(t)=2$. Let $K$ be any maximal subgroup of $R$ containing $N(R)$ and with $t\notin K$, and let $\chi:R\to {\bf Z}(R)$ be the surjective homomorphism with kernel $K$. Then Lemma~\ref{chi} produces a group automorphism $\phi_\chi:R\to R$, and since $H\le N(R)\le K$, $\phi_\chi$ is a power automorphism. However, $t^{\phi_\chi}=tz\notin \langle t\rangle$, as $z\notin\langle t\rangle$, where $z$ is the involution in ${\bf Z}(R)$. Since this is a contradiction, we have established~\eqref{eq:10}.

We are finally ready to finish the proof. Let $t\in R\setminus N(R)$. From~\eqref{eq:10}, $t$ has order $4$ and $t^2\in N(R)$. Assume towards a contradiction that $t^2=z$, where $z$ is the generator of ${\bf Z}(R)$. Since $t\notin N(R)={\bf C}_R(H)$, there exists $h\in H$ with $[h,t]=z$. Now, $th\notin N(R)$ and 
$$(th)^2=th\cdot th=t^2h[h,t]h=zhzh=z^2h^2=1,$$
contradicting that all elements in $R\setminus N(R)$ have order $4$. Thus $t^2\notin{\bf Z}(R)$. Now, let $K$ be any maximal subgroup of $R$ with $N(R)\le K$ and $t\notin K$ and let $\chi:R\to {\bf Z}(R)$ be the surjective group homomorphism with kernel $K$. Then Lemma~\ref{chi} produces a group automorphism $\phi_\chi:R\to R$, and since $H\le N(R)\le K$, it is a power automorphism. But $\phi_\chi$ is not a power automorphism because $t^{\phi_\chi}=tz\notin \langle t\rangle$, as $z\ne t^2$. This final contradiction concludes the proof.

\section{Future directions}

A natural next step for this work would be attempting to characterize groups $R$ such that the base size of the action of $\mathrm{Aut}(R)$ on $\mathrm{Sub}(R)$ is 2. We have made a variety of attempts to consider this. Computational results suggest that this family of groups is larger than one might intuitively anticipate, but our algorithms have been unable to consider large enough groups even to suggest a classification that we might then hope to prove. The methods we have used in this paper do have some relevance to the context of base size 2, but it seems that some additional ideas would be required.

If general results on base sizes greater than 1 seem beyond our reach, it would also be of interest to study the question of the base size of this action on some more specific families of groups. We believe that results of Burness and coauthors, for example~\cite{Burness1,Burness2}, might be useful for groups close to being simple. For instance, one can easily show that, for every non-abelian simple group $R$, the base size of $\mathrm{A}(R)$ is at most $3$. Here we do not use~\cite{Burness1,Burness2}; instead, to make the proof even shorter, we use a result about generation~\cite{Malle}. Indeed, when $R=\mathrm{PSU}_3(3)$ the result follows from a computer computation. Assume then that $R\ne \mathrm{PSU}_3(3)$. By~\cite[Theorem~A]{Malle}, $R$ is generated by three involutions $x_1,x_2,x_3$. It is now clear that
$${\bf N}_{\mathrm{Aut}(R)}(\langle x_1\rangle)\cap {\bf N}_{\mathrm{Aut}(R)}(\langle x_2\rangle)\cap{\bf N}_{\mathrm{Aut}(R)}(\langle x_3\rangle)
={\bf C}_{\mathrm{Aut}(R)}(\langle x_1,x_2,x_3\rangle)
={\bf C}_{\mathrm{Aut}(R)}(R)=1.$$
Hence $\mathrm{A}(R)$ has base size at most $3$.

In the case of abelian groups, the work of Babai, Goodman, and Lov\'{a}sz can be used to show that the base size of any abelian group is at most $5$. In~\cite[Lemma 6.1]{Babai1}, they define 4 subgroups of any abelian group and in Lemma 6.3 show (effectively) that if some group automorphism fixes each of these subgroups setwise and the element $x_1$ is also fixed, then every element of the group is fixed. If instead of fixing $x_1$ pointwise we fix $\langle x_1\rangle$ setwise (in addition to the other 4 subgroups), then essentially the same proof is enough to show that power automorphisms are possible, but nothing else. 

Of course, Babai's conjecture about lattice representations of groups (our original motivation for this work) remains open and of great interest. As noted previously, we have achieved some extensions of previous work on this (especially on groups of odd order) but the general problem remains challenging, particularly in the context of $2$-groups.


\thebibliography{10}
\bibitem{5}L.~Babai, \textit{Finite digraphs with given regular automorphism groups}, Period. Math. Hungar. \textbf{11} (1980), no. 4, 257--270.
\bibitem{4}L.~Babai, \textit{On the abstract group of automorphisms}, Combinatorics, Proceedings 8th British Combinatorics Conference, Swansea 1981, ed. H. N. V. Temperley; London Mathematical Society Lecture Notes 52 (University Press, Cambridge, 1981) 1--40.
\bibitem{Babai0}L.~Babai, A.~J.~Goodman, \textit{Subdirectly reducible groups and edge-minimal graphs with given automorphism group}, J.~London Math. Soc. \textbf{47} (1993), 417--432.
\bibitem{Babai1}L.~Babai, A.~J.~Goodman, L.~Lov\'{a}sz, \textit{Graphs with given automorphism group and few edge orbits}, European J. Combin. \textbf{12} (1991), 185--203. 
\bibitem{Baer1}R.~Baer, \textit{Der Kern, eine charakteristische Untergruppe}, Compositio Math. \textbf{1} (1935), 254--283.
\bibitem{birkhoff}G.~Birkhoff, \textit{Sobre los grupos de automorfismos}, Rev. Un. Mat. Argentina \textbf{11} (1946), 155--157.
\bibitem{Burness1}T.~C.~Burness, M.~Garonzi, A.~Lucchini, \textit{Finite groups, minimal bases and the intersection number}, Trans. Amer. Math. Soc. \textbf{9} (2022), 20--55.
\bibitem{Burness2}T.~C.~Burness, H.~Y.~Huang, \textit{On the intersection of Sylow subgroups in almost simple groups}, J.~Algebra \textbf{690} (2026), 596--631.
\bibitem{Dedekind}R.~Dedekind, \textit{\"{U}ber Gruppen, deren s\"{a}mmtliche Theiler Normaltheiler sind}, Mathematische Annalen \textbf{48} (1897), 548--561.
\bibitem{frucht}R.~Frucht, \textit{Herstellung von Graphen mit vorgegebener abstrakter Gruppe}, Compositio Math. \textbf{6} (1938), 239--250.
\bibitem{Goodman}A.~J.~Goodman, \textit{The edge-orbit conjecture of Babai}, J.~Combin. Theory Ser. B \textbf{57} (1993), 26--35.
\bibitem{Malle}G.~Malle, T.~Weigel, \textit{Generation of classical groups}, Geometriae Dedicata \textbf{49} (1995), 85--116.
\bibitem{Schenkman}E.~Schenkman, \textit{On the norm of a group}, Illinois J. Math. \textbf{4} (1960), 150--152.
\end{document}